\documentclass[11pt,a4paper]{amsart}
\usepackage{amsmath,amsfonts,amsthm,amsopn,color,amssymb,enumitem}
\usepackage[a4paper]{geometry}
\usepackage{palatino}
\usepackage{graphicx}
\usepackage[colorlinks=true]{hyperref}
\hypersetup{urlcolor=blue, citecolor=red, linkcolor=blue}

\usepackage{relsize}
\usepackage{esint}
\usepackage{verbatim}
\usepackage{mathrsfs}
\usepackage{xcolor}

\newcommand{\e}{\varepsilon}

\newcommand{\R}{\mathbb{R}}

\newcommand{\de}{\partial}

\renewcommand{\d }{\delta }

\newcommand{\D}{{\mathfrak{D}}}

\newcommand{\dsh}{{2^\sharp}}
\newcommand{\dst}{{2^*}}

\newcommand{\Ud}{{\mathcal{U}_{p, \delta}}}
\newcommand{\Vd}{\mathcal{V}_{p, \delta}}
\newcommand{\Wd}{\mathcal{W}_{p, \delta}}

\newcommand*{\abs}[1]{\left\vert #1\right\vert}

\newtheorem{theorem}{Theorem}[section]
\newtheorem*{theorem*}{Theorem}
\newtheorem{lemma}[theorem]{Lemma}

\newtheorem{proposition}[theorem]{Proposition}

\theoremstyle{definition}

\title[Positive blow-up solutions for a linearly perturbed boundary Yamabe problem]{Positive blow-up solutions for a linearly perturbed boundary Yamabe problem}

\author{Sergio Cruz-Blázquez}
\address{Sergio Cruz Blázquez, Dipartimento di Matematica, Università degli Studi di Bari, Via E. Orabona, 4, 70125 Bari (Italy)}
\email{sergio.cruz@uniba.it}

\author{Giusi Vaira}
\address{Giusi Vaira, Dipartimento di Matematica, Università degli Studi di Bari, Via E. Orabona, 4, 70125 Bari (Italy)}
\email{giusi.vaira@uniba.it}
\date\today
\subjclass[2010]{35B44, 53C21, 58J32}
\keywords{Prescribed curvature problem, conformal metric, Blow-up solutions.}
\thanks{S. Cruz and G. Vaira have been supported by INdAM – GNAMPA Project 2022 “Fenomeni di blow-up per equazioni non lineari”, E55F22000270001S and by PRIN 2017JPCAPN}

\begin{document}
\maketitle

\begin{abstract}
We consider the problem of prescribing the scalar and boundary mean curvatures via conformal deformation of the metric on a $n-$ dimensional compact Riemannian manifold. We deal with the case of negative scalar curvature $K$ and boundary mean curvature $H$ of arbitrary sign which are non-constant and $\mathfrak D_n=\sqrt{n(n-1)}H\abs{K}^{-1\over 2}>1$ at some point of the boundary. It is known that this problem admits a positive mountain pass solution if $n=3$, while no existence results are known for $n\geq 4$. We will consider a perturbation of the geometric problem and show the existence of a positive solution which blows-up at a boundary point which is critical for both prescribed curvatures.
\end{abstract}

\section{Introduction}
\noindent Let $(M, g)$ be a compact Riemannian manifold of dimension $n\geq 4$ with boundary $\partial M$. A problem that has drawn considerable attention in the literature is the following: {\it given two smooth functions $K$ and $H$, find a metric which is conformal to $g$ whose scalar curvature is $K$ and boundary mean curvature is $H$.}\\
It is well-known that, if $\tilde g= u^{\frac{4}{n-2}}g$, for some smooth function $u>0$ (called the conformal factor) then $$K=u^{-\frac{n+2}{n-2}}\left(-\frac{4(n-1)}{n-2}\Delta_g u +\mathcal S_g u\right) \,\, \hbox{and}\,\, H=u^{-\frac{n}{n-2}}\left(\frac{2}{n-2}\frac{\partial u}{\partial \nu}+h_g u\right),$$ where $\mathcal S_g$ and $K$ are the scalar curvatures of $g$ and $\tilde g$ respectively, $h_g$ and $H$ denote the boundary mean curvatures, $\Delta_g$ is the Laplace-Beltrami operator and $\nu$ is the outward unit normal vector to $\partial M$. \\ Therefore, the geometric problem can be formulated as the problem of finding a positive solution of the following doubly critical PDE:
\begin{equation}\label{pb0}
\left\{\begin{aligned}&-\frac{4(n-1)}{n-2}\Delta_g u +\mathcal S_g u =K u^{\frac{n+2}{n-2}}\quad&\mbox{in}\,\, M\\
&\frac{2}{n-2}\frac{\partial u}{\partial \nu}+h_g u = H u^{\frac{n}{n-2}}\quad &\mbox{on}\,\, \partial M\end{aligned}\right.\end{equation}
The first paper dealing with \eqref{pb0} is \cite{C}, in which Cherrier established first criterions for the existence of solutions, similar to those available for the classical Yamabe problem on closed manifolds. Successively, Escobar \cite{E1}, Marques \cite{M1, M2}, Mayer \& Ndiaye \cite{MN} and Almaraz \cite{A} provided existence results for the scalar flat case (i.e. $K=0$ and $H$ a constant). Similarly, Escobar in \cite{E2}, Brendle \& Chen in \cite{BC} treated the case of minimal boundary (i.e. $H=0$ and $K$ a constant). \\ The presence of both curvatures was considered first by Escobar in \cite{E4}, then by Han \& Li in \cite{HY1, HY2} and Chen, Ruan \& Sun in \cite{crs}. In those papers the authors worked with constant $K$ and $H$ and $K>0$. \\ 
We would like to point out that the existence results obtained depend on the dimension of the manifold, the properties of the boundary (being or not umbilic) and the vanishing properties of the Weyl tensor (i.e. being identically zero or not on the boundary or on the whole manifold).\\

The case of variable functions has been studied in specific situations. More precisely Ben Ayed,  El Mehdi \& Ould Ahmedou \cite{BEO1, BEO2} and Li \cite{L} studied the case $H=0$ on the half-sphere.  Abdelhedi,  Chtioui \&  Ould Ahmedou \cite{ACA}, Chang,  Xu  \& Yang \cite{CXY}, Djadli,  Malchiodi \& Ould Ahmedou \cite{DMA2} and Xu \& Zhang \cite{XZ} considered the case  $K=0$ and $M$ the unit ball.
The case when both $K$ and $H$ do not vanish has been studied by
Ambrosetti,  Li  \& Malchiodi  \cite{AYM} in a perturbative setting on the $n-$dimensional unit ball  and by Djadli,  Malchiodi \& Ould Ahmedou \cite{DMA1} on the three-dimensional half sphere (here $K>0$).
Finally, we quote the result of  Chen,  Ho \& Sun \cite{CHS} where they found a solution to \eqref{pb0} when $H$ and $K$ are negative functions
provided the manifold is of negative Yamabe invariant.\\
The case $K<0$ and $H$ of arbitrary sign is recently studied in \cite{CMR} by Cruz-Bl\'azquez, Malchiodi \& Ruiz on manifolds of non-positive Yamabe invariant ($\mathcal S_g\leq 0$) equipped with the Escobar metric, that is, a conformal metric $g$ such that $h_g=0$ and $\mathcal S_g$ is a constant-sign function.
\\ They definede the {\it scaling invariant} function $\mathfrak D_n:\partial M \to \mathbb R$ by $$\mathfrak D_n(x)=\sqrt{n(n-1)}\frac{H(x)}{\sqrt{|K(x)|}}$$
and they established the existence of a solution to \eqref{pb0} whenever  $\mathfrak D_n<1$ along the whole boundary. On the other hand, if $\mathfrak D_n>1$ at some boundary point, they got a solution for generic choices of $K$ and $H$ only on three dimensional manifolds by using a mountain pass argument. To recover compactness a careful blow-up analysis was needed. Indeed, they showed that the blow-up occurs at point $p\in\partial M$ such that $\mathfrak D_n(p)\geq 1$. To discard the blow-up around singular points with $\mathfrak D_n>1$ an \textit{isolated and simpleness} condition was needed, which they could only prove in dimension three.
This restriction on the dimension of $M$ is not technical. In fact, three is the maximal dimension for which one can prove that this type of blow-ups are isolated and simple for generic choices of $K$ and $H$. Some partial results are available in higher dimensions (see \cite{AA, BAGOB, ABA}), but a general description of the issue is still missing. 

In light of the above, it is natural to ask:  \begin{center}{\it when $K<0$ and $H$ is such that $\mathfrak D_n>1$ somewhere on $\de M$, what about the existence of solutions for $n\geq 4$?}\\\end{center} 

In this paper, we answer positively to this question in a perturbative setting, i.e. let us consider the Escobar metric (namely $h_g=0$ ) and let $\mathcal S_g>0$. We are led to study a linear perturbation of the geometric problem \eqref{pb0} \begin{equation}\label{pb}
\left\{\begin{aligned}&-\frac{4(n-1)}{n-2}\Delta_g u +\mathcal S_g u =K u^{\frac{n+2}{n-2}}\quad&\mbox{in}\,\, M\\
&\frac{2}{n-2}\frac{\partial u}{\partial \nu}+\varepsilon u = H u^{\frac{n}{n-2}}\quad &\mbox{on}\,\, \partial M\end{aligned}\right.\end{equation} 
where $\varepsilon$ is a small and positive parameter.\\
This problem has been recently studied in \cite{CPV} for $K<0$ and $H>0$ constants. In this work a clustered blow-up solution (i.e. a solution with non-simple blow-up points) for \eqref{pb} is constructed in dimensions $4\leq n\leq 7$, giving a partial counterpart to the blow-up analysis in \cite{CMR}. \\

Here we aim to show the existence of a solution of \eqref{pb} that blows-up as a single bubble as $\e\to 0$ around a boundary point with $\D_n>1$ which is also critical for $H$ and $K$ when $n\geq 4$. More precisely, we assume
\begin{itemize}
\item[(Hyp)] We let $K, H$ be sufficiently regular functions such that $K<0$, $H$ is of arbitrary sign and there exists $p\in \de M$ with $\D_n(p)>1$. Assume further that $p$ is a common local minimum point which is non-degenerate, i.e. $\nabla K(p)=\nabla H(p)=0$ and $D^2 H(p)$ and $D^2K(p)$ are positive definite.
\end{itemize}
The main result is the following.
\begin{theorem}\label{principale}
Let $n\geq 4$ and let us assume (Hyp). Let the trace-free second fundamental form of $\partial M$ non-zero everywhere. Then for $\e>0$ small, there exists a positive solution $u_\e$ of \eqref{pb} such that $u_\e$ blows-up at $p\in \partial M$ as $\e\to 0$.
\end{theorem}
In order to prove Theorem \ref{principale} we employ a classical reduction process. \\ For the best of our knowledge this is the first work in which the perturbed problem \eqref{pb} is considered with nonconstant functions $K$ and $H$. In that regard, we want to remark that the presence of the two functions makes the proof of Theorem \ref{principale} rather intricate. Furthermore, it provides one of the first existence results (albeit in a perturbative setting) for the problem of prescribing non-constant curvatures. The case of constant functions is already contained in \cite{CPV}.\\

The rest of the paper is organized as follows. In Section \ref{pre} some preliminaries are needed. In Section \ref{ansatz} we find the good ansatz of our solution. Then, in Sections \ref{S-AUX} and \ref{BIF} we solve our problem and we give a proof of the main Theorem. 

\section{Notation and Preliminaries}\label{pre}
We endorse the Sobolev space $H^1_g(M)$ the scalar product $$\langle u, v\rangle :=\int_M \left(c_n \nabla_g u\nabla_g v + \mathcal S_g uv\right)\, d\nu_g$$ where $d\nu_g$ is the volume element of the manifold and $c_n:=\frac{4(n-1)}{(n-2)}$.\\
We let $\|\cdot\|$ be the norm induced by $\langle \cdot, \cdot\rangle$. Moreover, for any $u\in L^q(M)$ and $v\in L^q(\partial M)$, we put $$\|u\|_{L^q(M)}:=\left(\int_M |u|^q\, d\nu_g\right)^{\frac 1 q} \quad\text{and}\quad \|v\|_{L^q(\partial M)}:=\left(\int_{\partial M}|v|^q\, d\sigma_g\right)^{\frac 1q},$$ where $d\sigma_g$ is the volume element of $\partial M$. For notational convenience, we will often omit the volume or surface elements in integrals. 

\smallskip 

We have the well-known embedding continuous maps
$$\begin{aligned} &\mathfrak i_{\partial M}: H^1_g(M)\to L^{\frac{2(n-1)}{n-2}}(\partial M),\qquad\qquad & \mathfrak i_M: H^1_g(M) \to L^{\frac{2n}{n-2}}(M),\\ &\mathfrak i^*_{\partial M}: L^{\frac{2(n-1)}{n}}(\partial M)\to H^1_g(M),\qquad\qquad & \mathfrak i^*_M: L^{\frac{2n}{n+2}}(M)\to H^1_g(M).\end{aligned}$$
Now, given $\mathfrak f\in L^{\frac{2(n-1)}{n}}(\partial M)$, the function $w_1=\mathfrak i^*_{\partial M}(\mathfrak f)$ in $H^1_g(M)$ is defined as the unique solution of the equation \begin{equation*}\left\{\begin{aligned}&-c_n\Delta_g w_1+\mathcal  S_g w_1=0\quad &\mbox{in}\,\, M,\\ &\frac{\partial w_1}{\partial \nu}=\mathfrak f\quad &\mbox{on}\,\, \partial M.\end{aligned}\right.\end{equation*}
Similarly, if we let $\mathfrak g\in L^{\frac{2n}{n+2}}(M)$, $w_2=\mathfrak i^*_M(\mathfrak g)$ denotes the unique solution of the equation
\begin{equation*}\left\{\begin{aligned}&-c_n\Delta_g w_2+\mathcal S_g w_2=\mathfrak g\quad &\mbox{in}\,\, M,\\ &\frac{\partial w_2}{\partial \nu}=0\quad &\mbox{on}\,\, \partial M.\end{aligned}\right.\end{equation*}
By continuity of $\mathfrak i_M$ and $\mathfrak i_{\partial M}$ we get
$$\|\mathfrak i^*_{\partial M}(\mathfrak f)\|\leq C_1 \|\mathfrak f\|_{L^{\frac{2(n-1)}{n}}(\partial M)}\quad\text{and}\quad \|\mathfrak i^*_M(\mathfrak g)\|\leq C_2 \|\mathfrak g\|_{L^{\frac{2n}{n+2}}(M)},$$ for some $C_1>0$ and independent of $\mathfrak f$ and some $C_2>0$ and independent of $\mathfrak g$. Then, we are able to rewrite the problem \eqref{pb} as \begin{equation}\label{pb1} u=\mathfrak i^*_M(K\mathfrak g(u))+\mathfrak i^*_{\partial M}\left(\frac{n-2}{2}H \mathfrak f(u)-\varepsilon u \right),\end{equation} where we set $\mathfrak g(u):=(u^+)^{\frac{n+2}{n-2}}$ and $\mathfrak f(u)=(u^+)^{\frac{n}{n-2}}$.

\smallskip 

We also define the energy $J_\varepsilon: H^1_g(M)\to\mathbb R$ 
\begin{equation}\label{energia}\begin{aligned}
J_\varepsilon(u)&:=\int_M \left(\frac{c_n}{2}|\nabla_g u|^2+\frac 12\mathcal S_g u^2-K \mathfrak G(u)\right)\,d\nu_g-c_n\frac{n-2}{2}\int_{\partial M} H \mathfrak F(u)\, d\sigma_g\\&+(n-1)\varepsilon\int_{\partial M} u^2\,d\sigma_g,\end{aligned}\end{equation}
being $$\mathfrak G(s)=\int_0^s \mathfrak g(t)\, dt,\qquad \mathfrak F(s)=\int_0^s \mathfrak f(t)\,dt.$$ Notice that the critical points of \eqref{energia} are weak solutions to the problem \eqref{pb}.

\smallskip

The aim is to build solutions to \eqref{pb1} that blows-up at a boundary point as $\varepsilon\to 0^+$ for $n\geq 4$.\\ 
Let $p\in\partial M$ with $\D_n(p)>1$. The main ingredient to cook up our solution is the so-called \textit{bubble}, whose expression is given by
\begin{equation*}
U_{\delta, x_0(\delta)}(x):=\frac{\alpha_n}{|K(p)|^{\frac{n-2}{4}}}\frac{\delta^{\frac{n-2}{2}}}{\left(|x-x_0(\delta)|^2-\delta^2\right)^{\frac{n-2}{2}}}
\end{equation*} 
where $\alpha_n:=\left(4n(n-1)\right)^{\frac{n-2}{4}}$, $x_0(\delta):=$ $\left(\tilde x_0,-\mathfrak D_n(p)\delta\right) \in \mathbb R^n$, $\tilde x_0\in \R^{n-1}$ and $\delta>0$. When $\D_n(p)>1$, the $n-$dimensional family of functions defined above describe all the solutions to the following problem in $\R^n_+$ (see \cite{cfs}):
\begin{equation*}
\left\{\begin{aligned}& -c_n \Delta U =-|K(p)|U^{\frac{n+2}{n-2}}\quad &\mbox{in}\,\, \mathbb R^n_+\\ &\frac{2}{n-2}\frac{\partial U}{\partial\nu}=H(p) U^{\frac{n}{n-2}}\quad &\mbox{on}\,\,\partial\mathbb R^n_+.\end{aligned}\right.\end{equation*}
We set \begin{equation}\label{U}U(x)=U_{1, x_0(1)}(\tilde x, x_n)=\frac{\alpha_n}{|K(p)|^{\frac{n-2}{4}}}\frac{1}{\left(|\tilde x|^2+(x_n+\mathfrak D_n(p))^2-1\right)^{\frac{n-2}{2}}},\end{equation} where $\tilde x=(x_1, \ldots, x_{n-1})\in \mathbb R^{n-1}$ and $x_n>0$.\\
We also need to introduce the linear problem 
\begin{equation}\label{lineare}
\left\{\begin{aligned}&-c_n\Delta v +|K(p)|\frac{n+2}{n-2}U^{\frac{4}{n-2}}v=0\quad &\mbox{in}\,\, \mathbb R^n_+\\ &\frac{2}{n-2}\frac{\partial v}{\partial\nu}-\frac{n}{n-2}H(p)U^{\frac{n}{n-2}}v=0\quad &\mbox{on}\,\, \partial\mathbb R^n_+.\end{aligned}\right.\end{equation}
In Section 2 of \cite{CPV} we have shown that the $n-$ dimensional space of solutions of \eqref{lineare}  is generated by the functions 
\begin{equation}\label{Ji}
\mathfrak j_i(x):=\frac{\partial U}{\partial x_i}(x)=\frac{\alpha_n}{|K(p)|^{\frac{n-2}{4}}}\frac{(2-n) x_i}{\left(|\tilde x|^2+(x_n+\mathfrak D_n(p))^2-1\right)^{\frac{n}{2}}},\quad i=1, \ldots, n-1
\end{equation}
and
\begin{equation}\label{Jn}\begin{aligned}
\mathfrak j_n(x)&:=\left(\frac{2-n}{2}U(x)-\nabla U(x)\cdot (x+\mathfrak D_n(p)\mathfrak e_n)+\mathfrak D_n(p)\frac{\partial U}{\partial x_n}\right)\\
&=\frac{\alpha_n}{|K(p)|^{\frac{n-2}{4}}}\frac{n-2}{2}\frac{|x|^2+1-\mathfrak D_n(p)^2}{\left(|\tilde x|^2+(x_n+\mathfrak D_n(p))^2-1\right)^{\frac{n}{2}}}.\end{aligned}\end{equation}
As we will see, bubbles are not a good enough approximating solution, and need to be corrected by a higher order term $V_p:\mathbb R^n_+\to \mathbb R$ whose main properties are collected in the next proposition (see \cite[Proposition 3.1]{CPV} for the proof).
\begin{proposition}\label{vp} Let $U$ be as in \eqref{U} and set
	$$
		\mathtt E_p(x)=\sum\limits_{i,j=1}^{n-1}\frac{8(n-1)}{n-2}h^{ij}(p)\frac{\de^2 U(x)}{\de x_i \de x_j}x_n,\ x\in \mathbb R^n_+
	$$
where $h^{ij}(p)$ are the coefficients of the second fundamental form of $M$ at the point $p\in\de M$. Then the problem
\begin{equation*}
	\left\lbrace \begin{array}{ll}
		-\frac{4(n-1)}{n-2}\Delta V + \frac{n+2}{n-2}\abs{K(p)}U^\frac{4}{n-2}V=	\mathtt E_p & \text{ in } \R^n_+, \\[0.15cm]
		\frac{2}{n-2}\frac{\de V}{\de \nu} - \frac{n}{n-2}H(p) U^\frac{2}{n-2}V=0 & \text{ on } \de \R^n_+,
	\end{array}\right.
\end{equation*}
admits a solution $V_p$ satisfying the following properties:
\begin{enumerate}
	\item[(i)] $\displaystyle\int\limits_{\R^n_+}V_p(x)\mathfrak z_i(x)dx=0$ for any $ i=1,\ldots,n$ (see \eqref{Ji} and \eqref{Jn}),
	\item[(ii)] $\abs{\nabla^\alpha V_p}(x) \lesssim \frac1{\left(1+\abs{x}\right)^{n-3+\alpha}}$ for any  $x\in\R^n_+ $ and $\alpha = 0,1,2,$
	\item[(iii)]  
	\begin{equation*}
		|K(p)|\int_{\R^n_+} U^\frac{n+2}{n-2}V_pdx = (n-1)H(p)\int_{\de\R^n_+} U^\frac{n}{n-2}V_p\,d\tilde x,
	\end{equation*}
	\item[(iv)]   if  $n\geq 5$ $$\int_{\R^n_+}\left(-\frac{4(n-1)}{n-2}\Delta V_p+\frac{n+2}{n-2}\abs{K}U^\frac{4}{n-2}V_p\right)V_p\geq 0,\quad\text{and}$$	
	\item[(v)] the map $p\mapsto V_p$ is $C^2(\de M)$. 
\end{enumerate}
\end{proposition}

Here and in the sequel we agree that $f\lesssim g$  means $|f|\leq C |g| $ for some positive constant $c$ which is independent on $f$ and $g$ and $f\sim g$  means $f= g(1+o(1))$. We use the letter $C$ to denote a positive constant that may change from line to line.
\section{The ansatz}\label{ansatz}
We are now able to define the good ansatz of the solution we are looking for. To this aim, take $p\in\partial M$ with $\D_n(p)>1$ and consider $\psi_p^\partial: \mathbb R^n_+\to M$ the Fermi coordinates in a neighborhood of $p$. Then, let us define 
\begin{equation*}
\begin{aligned}\mathcal W_{p,\delta}(\xi)&:=\chi\left(\left(\psi_{p}^\partial\right)^{-1}(\xi)\right)\left[\frac{1}{\delta^{\frac{n-2}{2}}}U\left(\frac{\left(\psi_{p}^\partial\right)^{-1}(\xi)}{\delta}\right)+\delta\frac{1}{\delta^{\frac{n-2}{2}}}V_p\left(\frac{\left(\psi_{p}^\partial\right)^{-1}(\xi)}{\delta}\right)\right],\end{aligned}\end{equation*}
where $\chi$ is a radial cut-off function with support in a ball of radius $R$.

\smallskip

We look for solutions of \eqref{pb} or \eqref{pb1} of the form 
\begin{equation}\label{ue}u_\varepsilon(\xi):=\Wd(\xi)+\Phi_\varepsilon(\xi)\quad\xi\in M,\end{equation}
where the concentration parameter $\delta$ is choosen as follows:
\begin{equation*}
\delta:=\varepsilon d,\quad   d\in(0, +\infty)\,\quad \hbox{if}\ n\geq5,
\end{equation*}
or   
\begin{equation*}
\delta:=\rho(\e)d,\  d\in(0, +\infty) \  
 \quad \hbox{if}\ n=4,
\end{equation*}
where $\rho$  is the inverse function of 
$\ell:(0, e^{-\frac 12})\to \left(0, \frac{e^{-1}}{2}\right)$ defined by $\ell(s)=-s\ln  s$ and $d$ is chosen in \eqref{choiced}. We remark that $\rho(\e)\to 0$  as $\e\to 0$.

\smallskip

The remainder term $\Phi_\varepsilon(\xi)$ belongs to $\mathcal K^\bot$ and is defined as follows. For any $i=1, \ldots, n$ we define the functions
\begin{equation*}
\mathcal Z_{i}(\xi):=\frac{1}{\delta^{\frac{n-2}{2}}}\mathfrak j_i\left(\frac{\left(\psi_p^\partial\right)^{-1}(\xi)}{\delta}\right)\chi\left(\left(\psi_p^\partial\right)^{-1}(\xi)\right)\,\, i=1, \ldots, n,\end{equation*} where $\mathfrak j_i$ are defined in \eqref{Ji} and \eqref{Jn}.\\
We decompose $H^1_g(M)$ in the direct sum of the following two subspaces $$\mathcal K={\rm span}\left\{\mathcal Z_{i}\,\,:\,\, i=1, \ldots, n\,\, \right\}$$ and $$\mathcal K^\bot:=\left\{\varphi\in H^1_g(M)\,:\, \langle \varphi, \mathcal Z_{i}\rangle=0,\quad i=1, \ldots, n\,\, \right\},$$ and we define the corresponding projections $$\Pi: H^1_g(M)\to \mathcal K,\qquad \Pi^\bot: H^1_g(M)\to \mathcal K^\bot.$$
Therefore solving \eqref{pb1} is equivalent to solve the system
\begin{equation}\label{aux}\Pi^\bot\left\{u_\varepsilon -\mathfrak i^*_M(K \mathfrak g(u_\varepsilon))-\mathfrak i^*_{\partial M}\left(\frac{n-2}{2}H \mathfrak f(u_\varepsilon)-\varepsilon u_\varepsilon\right)\right\}=0, \end{equation}
\begin{equation}\label{bif}\Pi\left\{u_\varepsilon -\mathfrak i^*_M(K \mathfrak g(u_\varepsilon))-\mathfrak i^*_{\partial M}\left(\frac{n-2}{2}H \mathfrak f(u_\varepsilon)-\varepsilon u_\varepsilon\right)\right\}=0,\end{equation}  with $u_\varepsilon$ defined in \eqref{ue}.

\eqref{aux} and \eqref{bif} are called the \textit{auxiliar} and the \textit{bifurcation} equations, respectively.
\smallskip

Finally, we introduce some integral quantities that will appear in our computations, whose properties are listed in the Appendix of \cite{CPV}:
\begin{equation*}I_m^\alpha:=\int_0^{+\infty}\frac{\rho^\alpha}{(1+\rho^2)^m}\, d\rho,\quad \hbox{with}\, \alpha+1<2m\end{equation*}
and, for $p\in\partial M$ with $\D_n(p)>1$,
\begin{equation*}
 \varphi_m(p):=\int_{\mathfrak D_n(p)}^{+\infty}\frac{1}{(t^2-1)^m}\, dt\quad\hbox{and}\quad {\hat\varphi}_m(p):=\int_{\mathfrak D_n(p)}^{+\infty}\frac{(t-\mathfrak D_n(p))^2}{(t^2-1)^m}\, dt .\end{equation*}

\section{The Remainder Term: Solving Equation (\ref{aux})} \label{S-AUX}
In this section we find the remainder term $\Phi_\varepsilon\in\mathcal K^\bot$ via a fixed point argument. First, we rewrite the equation \eqref{aux} as 
\begin{equation}\label{aux1}
\mathcal E+\mathcal L(\Phi_\varepsilon)+\mathcal N(\Phi_\varepsilon)=0,\end{equation}
where the error term $\mathcal E$ is defined by
\begin{equation*}
\mathcal E:=\Pi^\bot\left\{\Wd -\mathfrak i^*_M\left(K \mathfrak g\left(\Wd \right)\right)-\mathfrak i^*_{\partial M}\left(\frac{n-2}{2} H \mathfrak f\left(\Wd \right)-\varepsilon \Wd \right)\right\},\end{equation*} 
the linear operator $\mathcal L$ is defined by 
\begin{equation*}
\mathcal L(\Phi_\varepsilon):=\Pi^\bot\left\{\Phi_\varepsilon-\mathfrak i^*_M\left(K \mathfrak g'\left(\Wd \right)\Phi_\varepsilon\right)-\mathfrak i^*_{\partial M}\left(\frac{n-2}{2}H \mathfrak f'\left(\Wd\right)\Phi_\varepsilon -\varepsilon  \Phi_\varepsilon\right)\right\}\end{equation*} and 
\begin{equation*}\begin{aligned}
\mathcal N (\Phi_\varepsilon)&:=\Pi^\bot\left\{-\mathfrak i^*_M\left[K\left(\mathfrak g\left(\Wd +\Phi_\varepsilon\right)-\mathfrak g\left(\Wd \right)-\mathfrak g'\left(\Wd\right)\Phi_\varepsilon\right)\right]\right.\\
&\left.-\mathfrak i^*_{\partial M}\left[\frac{n-2}{2}H\left(\mathfrak f\left(\Wd +\Phi_\varepsilon\right)-\mathfrak f\left(\Wd \right)-\mathfrak f'\left(\Wd \right)\Phi_\varepsilon\right)\right]\right\}.\end{aligned}\end{equation*}
Reasoning as in Proposition 4.1 of \cite{CPV} and using the fact that $p$ is a critical point of $H$ and $K$ we are able to establish the following result:
\begin{proposition}\label{esistphi}
Let $n\geq 4$. Let $p\in\partial M$ a critical point of $H$ and $K$. For any real numbers $a, b$ such that $0<a<b$,
there exists a positive constant $C=C(a, b)$ and $\e_0>0$  such that for any $\e\in (0, \e_0)$ and for any $d_0\in [a, b]$ there exists a unique function $\Phi_\e\in\mathcal K^\bot$ which solves equation \eqref{aux1}. 
Moreover, the map $d_0 \mapsto \Phi_\e (d_0)$ is of class $C^1$ and 
$$\|\Phi_\e\|\lesssim\left\{\begin{aligned}&\e^2 \quad &\hbox{if}\,\,\ n\geq 7\\ &\e^2|\ln \e|^{\frac 23} \quad &\hbox{if}\,\,\ n=6\\
& \e^{\frac 32}  \quad &\hbox{if}\,\,\ n=5\\
&\rho(\e)\quad &\hbox{if}\,\,\ n=4.\end{aligned}\right.
$$\end{proposition}

\section{The Reduced Energy and the Proof of Theorem \ref{principale}}\label{BIF}
Let us introduce the reduced energy 
\begin{equation}\label{RE}
\mathfrak J_\e(d, p):= J_\e(\mathcal W_{\delta, p}+\Phi_\e),\end{equation} where
the remainder term $\Phi_\e$ is given by Proposition \ref{esistphi}.\\ The following result allows as usual to reduce our problem to a finite dimensional one. 
\begin{proposition}\label{pc} The following assertions hold true:
\begin{itemize}
\item[(i)] $\mathcal W_{\delta, p}+\Phi_\e$ is a solution to \eqref{pb} if and only if  
 $(d, p)\in(0, +\infty)\times\partial M$ is a critical point of the reduced energy $\mathfrak J_\e(p, d)$ defined in \eqref{RE}.

\item[(ii)] The following expansion holds:
$$\mathfrak J_\e(d, p)=\mathfrak E(p)  -\zeta_n(\e) \left[ d^2\mathtt A_n(p)-d\mathtt C_n(p)\right]+o(\zeta_n(\e)),$$
 where $\mathfrak E(p)$ is the energy of the bubble evaluated in Lemma \ref{energiabubble}, $$\mathtt A_n(p):= \mathtt R(p)\|\pi(p)\|^2+(n-2)\int_{\mathbb R^{n-1}}\langle D^2 H(p)\tilde x, \tilde x\rangle U^\dsh\, d\tilde x +\frac{n-2}{2n}\int_{\mathbb R^n}\langle D^2 K(p)x, x\rangle U^\dst\, dx$$ and $$\mathtt C_n(p):=2(n-2)\omega_{n-1}\alpha_n^2\frac{1}{|K(p)|^{\frac{n-2}{2}}\left(\mathfrak D_n^2(p)-1\right)^{\frac{n-3}{2}}}I^{n}_{n-1}$$ are strictly positive functions and $\mathtt R(p)$ depends on $p$ but it is also positive. Moreover
$$\zeta_4(\e):=\rho^2(\e)|\ln \rho(\e)|\quad \hbox{and}\quad \zeta_n(\e):=\e^2\ \hbox{if}\ n\geq5.$$
 
\end{itemize}

\end{proposition}

\begin{proof} It is quite standard to prove that $$
J_\e (\mathcal W_{\delta, p}+\Phi_\e)=J_\e(\mathcal W_{\delta, p})+\mathcal O(\|\Phi_\e\|^2)=J_\e(\mathcal W_{\delta, p})+\left\{\begin{aligned} &\mathcal O(\e^4)\quad &\hbox{if}\,\, n\geq 7\\ &\mathcal O(\e^4|\ln\e|^{\frac 4 3})\quad &\hbox{if}\,\, n=6\\ &\mathcal O(\e^3)\quad &\hbox{if}\,\, n=5\\ &\mathcal O(\rho(\e)^2)\quad &\hbox{if}\,\, n=4\end{aligned}\right.$$ 
For the expansion of the reduced energy we follow the proof of Proposition 4.4 of \cite{CPV} so many details will be skipped for the sake of brevity. We let in the following $$\Ud=\chi\left(\left(\psi_{p}^\partial\right)^{-1}(\xi)\right)\frac{1}{\delta^{\frac{n-2}{2}}}U\left(\frac{\left(\psi_{p}^\partial\right)^{-1}(\xi)}{\delta}\right),\quad \Vd=\chi\left(\left(\psi_{p}^\partial\right)^{-1}(\xi)\right)\frac{1}{\delta^{\frac{n-2}{2}}}V_p\left(\frac{\left(\psi_{p}^\partial\right)^{-1}(\xi)}{\delta}\right).$$
We expand in $\delta$ the functional 
$$\begin{aligned}
 J_\e(\Wd)&:= \underbrace{\frac{c_n}{2}\int_M |\nabla_g(\Ud+\delta \Vd)|^2}_{I_1}+\underbrace{\frac 12 \int_M \mathcal S_g (\Ud +\delta \Vd)^2}_{I_2}+\underbrace{(n-1)\e\int_{\partial M} (\Ud +\delta\Vd)^2}_{I_3}\\
 &\underbrace{-(n-2)\int_{\partial M} H \left[\left((\Ud +\delta \Vd)^+\right)^{\frac{2(n-1)}{n-2}}-\Ud^{\frac{2(n-1)}{n-2}}\right]}_{I_4}\underbrace{-(n-2)\int_{\partial M}H \Ud^{\frac{2(n-1)}{n-2}}}_{I_5}\\
 &\underbrace{-\frac{n-2}{2n}\int_M K\left[\left((\Ud +\delta \Vd)^+\right)^{\frac{2n}{n-2}}-\Ud^{\frac{2n}{n-2}}\right]}_{I_6}\underbrace{-\frac{n-2}{2n}\int_M K \Ud ^{\frac{2n}{n-2}}}_{I_7}
 \end{aligned}$$
 {\it Estimate of $I_2$:} For $n\geq 5$ it holds:

 $$\begin{aligned}
 I_2&:=\frac 12\int_{\mathbb R^n_+}\mathcal S_g(\delta x)\left(U(x)\chi(\delta x)+\delta V_p(x)\chi(\delta x)\right)^2|g(\delta x)|^{\frac 12}\, dx\\
 &=\delta^2 \frac 12 \alpha_n^2 \omega_{n-1}\frac{2(n-2)}{n-1}I_{n-1}^n \frac{\mathcal S_g(p)}{|K(p)|^{\frac{n-2}{2}}}\varphi_{\frac{n-3}{2}}(p)+\left\{\begin{aligned}&\mathcal O(\delta^3)\quad &\hbox{if}\,\, n\geq 6\\
 &\mathcal O(\delta^3|\ln\delta|)\quad &\hbox{if}\,\, n=5,\end{aligned}\right. \end{aligned}$$
 while for $n=4$ we get
 $$\begin{aligned}
 I_2&= -\frac{2\alpha_4^2\omega_3}{3} \frac{\mathcal S_g(p)}{|K(p)|}I_3^4\d^2\ln \d +\mathcal O(\d^2).\end{aligned}$$
{\it Estimate of $I_3$:} For $n\geq 4$ we have
 
 $$\begin{aligned}
 I_3&:=(n-1)\e \delta\int_{\mathbb R^{n-1}}\left(U(\tilde x, 0)\chi(\delta\tilde x, 0)+\delta V_p (\tilde x, 0)\chi(\delta\tilde x, 0)\right)^2|g(\delta\tilde x, 0)|^{\frac 12}\, d\tilde x\\
 &=\e\d\underbrace{2(n-2)\omega_{n-1}\alpha_n^2\frac{1}{|K(p)|^{\frac{n-2}{2}}\left(\mathfrak D_n^2(p)-1\right)^{\frac{n-3}{2}}}I^{n}_{n-1}}_{:=\mathtt C_n(p)\geq 0}+\left\{\begin{aligned}&\mathcal O(\e \delta^2)\quad &\hbox{if}\,\, n\geq 5\\ &\mathcal O(\e\delta^2|\ln\d|)\quad &\hbox{if}\,\, n=4.\end{aligned}\right.
 \end{aligned}$$ 
  {\it Estimate of $I_5$:}
 Since $p$ is a non-degenerate minimum point of $H$, we get for $n\geq 4$ that
 $$\begin{aligned} I_5&:=-(n-2)\int_{\mathbb R^{n-1}} H(\delta\tilde x, 0)\left(U(\tilde x, 0)\chi(\delta\tilde x, 0)\right)^{\dsh}|g(\delta \tilde x, 0)|^{\frac 12}\, d\tilde x\\
 &=-(n-2)H(p)\int_{\partial \mathbb R^n_+}U^\dsh(\tilde x, 0)\, d\tilde x-(n-2)\d^2\int_{\partial \mathbb R^n_+}\langle D^2H(p)\tilde x, \tilde x\rangle U^\dsh(\tilde x, 0)\, d\tilde x\\
 &+\delta^2\frac{\alpha_n^{\frac{2(n-1)}{n-2}}(n-2)}{6(n-1)}\omega_{n-1}\frac{H(p)\overline R_{ii}(p)}{\left(\mathfrak D_n^2(p)-1\right)^{\frac{n-3}{2}}|K(p)|^{\frac{n-1}{2}}}I_{n-1}^n+\mathcal O(\d^3).
 \end{aligned}$$
   {\it Estimate of $I_7$:}
  Analogously, since $p$ is a non-degenerate minimum point of $K$, we get for $n\geq 4$
$$\begin{aligned} I_7&:=-\frac{n-2}{2n}\int_{\mathbb R^n_+} K(\delta x)\left(U(x)\chi(\delta x)\right)^{\dst}|g(\delta x)|^{\frac 12}\, dx\\
&=\frac{n-2}{2n}|K(p)|\int_{\mathbb R^n_+}U^{\dst}(x)\, dx-\frac{n-2}{2n}\d^2\int_{\mathbb R^n_+}\langle D^2 K(p)\tilde x, \tilde x\rangle U^\dst(x)\, dx\\
&-\delta^2\frac{n-2}{4n}\frac{n-3}{2(n-1)}I_{n-1}^n\omega_{n-1}\alpha_n^{\dst}\frac{\left(\|\pi(p)\|^2+{\rm Ric}_\nu(p)\right)}{|K(p)|^{\frac{ n-2}{ 2}}}\hat\varphi_{\frac{n+1}{2}}(p)\\
&-\delta^2\alpha_n^{\dst}\frac{n-3}{2(n-1)}I_{n-1}^n\omega_{n-1}\frac{n-2}{12n(n-1)}\frac{\overline{R}_{ii}(p)}{|K(p)|^{\frac{ n-2}{ 2}}}\varphi_{\frac{n-1}{2}}(p)+\mathcal O(\delta^3).
\end{aligned}$$
   {\it Estimate of $I_4$ and $I_6$:}
Using again the fact that $p$ is a critical point of $H$ and $K$ we get for $n\geq 4$ that
$$\begin{aligned} I_4&=-(n-2)\int_{\partial \mathbb R^n_+}H(\delta\tilde x, 0)\left[\left((U+\delta V_p)^+\right)^{\frac{2(n-1)}{n-2}}-U^{\frac{2(n-1)}{n-2}}\right](\tilde x, 0)|g(\delta\tilde x, 0)|^{\frac 12}\, d\tilde x\\
&=-2(n-1)\delta H(p)\int_{\partial \mathbb R^n_+}U^{\frac{n}{n-2}}V_p\, d\tilde x -\frac{n(n-1)}{n-2}\delta^2 H(p)\int_{\partial \mathbb R^n_+}U^{\frac{2}{n-2}}V_p^2\, d\tilde x+\mathcal O(\delta^3)\end{aligned}$$
and
 $$\begin{aligned}
 I_6&=|K(p)|\delta\int_{\mathbb R^n_+}U^{\frac{n+2}{n-2}}V_p+\frac{n+2}{2(n-2)}\delta^2|K(p)|\int_{\mathbb R^n_+}U^{\frac{4}{n-2}}V_p^2+\mathcal O(\delta^3).
 \end{aligned}$$
    {\it Estimate of $I_1$:}
 Now we evaluate $I_1$. First we write $$I_1:=\underbrace{\frac{c_n}{2}\int_M|\nabla_g \Ud|^2\,d\nu_g}_{:=I_1^1}+\underbrace{c_n\delta\int_M \nabla_g\Ud\nabla_g\Vd\, d\nu_g}_{:=I_1^2}+\underbrace{\frac{c_n}{2}\delta^2\int_M |\nabla_g\Vd|^2\, d\nu_g}_{:=I_1^3}$$ and we analyze separately the terms $I_1^i$ with $i=1, 2, 3$.
 
 \smallskip 
 
   {\it Estimate of $I_1^1$:} For $n\geq 5$
$$\begin{aligned}
 I_1^1 &=\frac{c_n}{2}\int_{\mathbb R^n_+}|\nabla U|^2-\frac{c_n\alpha_n^2(n-2)^2(n-3)}{4(n-1)}\omega_{n-1}I_{n-1}^n\d^2\frac{(\|\pi(p)\|^2+{\rm Ric}_\eta(p))}{|K(p)|^{\frac{n-2}{2}}}\hat\varphi_{\frac{n-1}{2}}(p)\\
 &-\frac{c_n\alpha_n^2(n-2)^2(n-3)}{8(n-1)}\omega_{n-1}I_{n-1}^n\d^2\frac{(\|\pi(p)\|^2+{\rm Ric}_\eta(p))}{|K(p)|^{\frac{n-2}{2}}}\hat\varphi_{\frac{n+1}{2}}(p)\\
 &-\frac{c_n\alpha_n^2(n-2)^2}{12(n-1)}\omega_{n-1}I_{n-1}^n\d^2\frac{\overline R_{\ell \ell}(p)}{|K(p)|^{\frac{n-2}{2}}}\varphi_{\frac{n-3}{2}}(p)-\frac{c_n\alpha_n^2(n-2)^2(n-3)}{24(n-1)^2}\omega_{n-1}I_{n-1}^n\d^2\frac{\overline R_{\ell\ell}(p)}{|K(p)|^{\frac{n-2}{2}}}\varphi_{\frac{n-1}{2}}(p)\\
&+\frac{c_n\alpha_n^2}{4}\frac{(n-2)^2(n-3)}{(n-1)^2}\omega_{n-1}I_{n-1}^n\delta^2\frac{\left(3\|\pi(p)\|^2+{\rm Ric}_\eta(p)\right)}{|K(p)|^{\frac{n-2}{2}}}\hat\varphi_{\frac{n-1}{2}}(p)+\mathcal O(\delta^3)\\
 \end{aligned}$$
and if $n=4$
 $$\begin{aligned}I_1^1
 &= \frac{c_4}{2}\int_{\mathbb R^4_+}|\nabla U|^2+\frac 13 c_4\omega_3\alpha_4^2I_3^4\frac{\left(\|\pi(p)\|^2+{\rm Ric}_\nu(p)\right)}{|K(p)|}\d^2\ln \d +\frac 1 9 c_4\omega_3 \alpha_4^2I_3^4 \frac{\overline R_{\ell\ell}(p)}{|K(p)|}\d^2\ln \d \\
 &-\frac{1}{54} c_4\omega_3\alpha_4^2\frac{\left(3\|\pi(p)\|^2+{\rm Ric}_\nu(p)\right)}{|K(p)|}I_3^4\d^2\ln \d+\mathcal O(\d^2).
\end{aligned}$$
    {\it Estimate of $I_1^2$ and $I_1^3$ for $n\geq 5$:}
We have
$$\begin{aligned} I_1^2&=-\delta \frac{n+2}{n-2}|K(p)|\int_{\mathbb R^n_+}U^{\frac{n+2}{n-2}}V_p\, dx+c_n \frac n 2 \delta H(p)\int_{\partial \mathbb R^n_+}U^{\frac{n}{n-2}}V_p\\
&-c_n\delta^2\int_{\mathbb R^n_+}|\nabla V_p|^2+c_n\frac n 2 \delta^2H(p)\int_{\partial\mathbb R^n_+}U^{\frac{2}{n-2}}V_p^2-\frac{n+2}{n-2}|K(p)|\delta^2\int_{\mathbb R^n_+}U^{\frac{4}{n-2}}V_p^2+\mathcal O(\delta^3)\end{aligned}$$ Moreover
$$I_1^3:=\frac{c_n}{2}\delta^2\int_{\mathbb R^n_+}|\nabla V_p|^2+\mathcal O(\delta^3).$$
 {\it Estimate of $I_1^2$ and $I_1^3$ for $n=4$:}
We have (see \cite{CPV})
$$I_1^2+I_1^3=-3\frac{64\pi^2}{|K(p)|}\|\pi(p)\|^2\delta^2|\ln\d|+\mathcal O(\d^2).$$

{\it Conclusion:}
\begin{itemize}
\item[(i)] the terms of order $\d$ cancel because of Proposition \ref{vp}-(iii)
\item[(ii)] the higher order terms which contain ${\rm Ric}_\nu(p)$ and $\overline R_{\ell\ell}(p)$ (di order $\d^2$ if $n\geq5$ and $\d^2|\ln\d|$ if $n=4$) cancel.
\item[(iii)] the terms with $\|\pi(p)\|^2$ for $n\geq 5$ are
$$\begin{aligned}\left(\ldots\ldots\right)&=\delta^2\alpha_n^2\frac{n-2}{n-1}\omega_{n-1}I_{n-1}^n\frac{\|\pi(p)\|^2}{|K(p)|^{\frac{n-2}{2}}}\left(\varphi_{\frac{n-3}{2}}(p)-(n-1)(n-3)\hat\varphi_{\frac{n+1}{2}}(p)\right.\\
&\quad\left.-(n-1)(n-3)\hat\varphi_{\frac{n-1}{2}}(p)+3(n-3)\hat\varphi_{\frac{n-1}{2}}(p)\right)\\
&=-\delta^2\underbrace{\alpha_n^2\frac{n-2}{n-1}\omega_{n-1}I_{n-1}^n\frac{\|\pi(p)\|^2}{|K(p)|^{\frac{n-2}{2}}}\left(4(n-3)\hat\varphi_{\frac{n-1}{2}}(p)+\varphi_{\frac{n-3}{2}}(p)\right)}_{:=\mathtt B_n(p)>0}\end{aligned}$$
and by Proposition B.2 of \cite{CPV} $$\frac12 \int\limits_{\mathbb R^n_+}\left(-c_n\Delta V_p+{n+2\over n-2}|K(p)|U^{4\over n-2}V_p\right)V_p=\frac 12 \mathtt F_n(p) \|\pi(p)\|^2$$
\item[(iv)] the terms with $\|\pi(p)\|^2$ for $n=4$ are
$$\left(\ldots\ldots\right)=-\left(\frac{192\pi^2}{|K(p)|}+\alpha_4^2\omega_3 I_3^4 \frac{1}{|K(p)|}\right)\|\pi(p)\|^2\d^2|\ln\delta|.$$\end{itemize}
Then for $n\geq 5$
$$\mathfrak J_\e(d, p)=E(U)-\d^2\mathtt A_n(p)(1+o(1))+\e \d \mathtt C_n(p)(1+o(1)) $$
where $$\begin{aligned}\mathtt A_n(p)&=\mathtt B_n(p)+\frac 12 \mathtt F_n(p)\|\pi(p)\|^2+(n-2)\int_{\mathbb R^{n-1}}\langle D^2 H(p)\tilde x, \tilde x\rangle U^\dsh(\tilde x, 0)\, d\tilde x\\& +\frac{n-2}{2n}\int_{\mathbb R^n}\langle D^2 K(p)x, x\rangle U^\dst(x)\, dx>0\end{aligned}.$$
While for $n=4$ we have $$\mathfrak J_\e(d, p)= E(U)-\d^2|\ln\d|\mathtt A_4(p)(1+o(1))+\e \d \mathtt C_4(p)(1+o(1)) $$
where $$\begin{aligned}\mathtt A_4(p)&=\left(\frac{192\pi^2}{|K(p)|}+\alpha_4^2\omega_3 I_3^4 \frac{1}{|K(p)|}\right)\|\pi(p)\|^2+2\int_{\mathbb R^{3}}\langle D^2 H(p)\tilde x, \tilde x\rangle U^\dsh(\tilde x, 0)\, d\tilde x \\
&+\frac{1}{4}\int_{\mathbb R^n}\langle D^2 K(p)x, x\rangle U^\dst\, dx.\end{aligned}$$
In both cases $\mathtt A_n(p)\geq 0$ since $p$ is a non-degenerate minimum point of $H$ and $K$. By the choice of $\delta$ the result follows.
\end{proof}
We are now ready to prove the main result of the paper:
\begin{proof}[Proof of Theorem \ref{principale}]
By Proposition \ref{pc} it is sufficient to find a critical point of the reduced energy $\mathfrak J_\e(p, d)$.

\smallskip

Let us set $\mathtt G_p(d)=d \mathtt C_n(p)-d^2 \mathtt A_n(p)$ where $\mathtt C_n(p)$ and $\mathtt A_n(p)$ are the positive functions defined in Proposition \ref{pc}. Let $p_0\in\partial M$ be a non-degenerate minimum point of $H$ and $K$ in the sense of the assumption (Hyp). Then it is easy to see that $p_0$ is a non-degenerate maximum point of $\mathfrak E(p)$.\\ Hence, there is a $\sigma_1-$ neighbourhood of $p_0$, say $\mathcal U_{\sigma_1}\subset \partial M$, such that for any sufficiently small $\gamma>0$ \begin{equation}\label{Ep}\mathfrak E(p)\leq \mathfrak E(p_0)-\gamma\quad \forall\,\, p\in \partial U_{\sigma_1}.\end{equation} Now we see that
\begin{equation}\label{choiced} d_0:=\frac{\mathtt C_n(p_0)}{2\mathtt A_n(p_0)}\end{equation}
is a strictly maximum point of the function $\mathtt G_{p_0}(d)$. Then there is an open interval $I_{\sigma_2}$ such that $\bar I_{\sigma_2}\subset \mathbb R^+$ and 
\begin{equation}\label{Gp}
G_{p_0}(d)\leq G_{p_0}(d_0)-\gamma\quad \forall\,\, d\in\partial I_{\sigma_2}.\end{equation}
Let us set $\mathcal K:=\overline{\mathcal U_{\sigma_1}\times I_{\sigma_2}}$ and let $\eta>0$ be small enough so that $\mathcal K\subset \mathcal U_{\sigma_1}\times \left(\eta, \frac 1 \eta\right)$. Since the reduced functional is continuous on $\mathcal K$ then, by Weierstrass Theorem it follows that it has a global maximum point in $\mathcal K$. Let $(p_\e, d_\e)$ such point. We want to show that it is in the interior of $\mathcal K$.\\
By contradiction suppose that the point $(p_{\e}, d_{\e})\in\partial \mathcal K$. There are two possibilities: \begin{itemize}\item[(a)] $p_{\e}\in\partial\mathcal U_{\sigma_1}$, $d_{\e}\in \bar I_{\sigma_2}$ \item[(b)]  $p_{\e}\in\mathcal U_{\sigma_1}$, $d_{\e}\in \partial I_{\sigma_2}$.\end{itemize}

If (a) holds, by using Proposition \ref{pc}, the fact that $(p_\e, d_\e)$ is a maximum point for $\mathfrak J$ and \eqref{Ep} we have 
\begin{equation*} 0\leq \mathfrak J_\e(p_\e, d_\e)-\mathfrak J_\e(p_0, d_\e)=\mathfrak E(p_\e)-\mathfrak E(p_0)+o(\zeta_n(\e))\leq -\gamma +o(\zeta_n(\e))<0\end{equation*} for $\e$ sufficiently small, which is in contradiction with \eqref{jp1}.\\ 
If now (b) holds, then by using Proposition \ref{pc}, again the fact that $(p_\e, d_\e)$ is a maximum point for $\mathfrak J_\e$ and \eqref{Gp} we have 
\begin{equation}\label{jp1}
0\leq \mathfrak J_\e(p_\e, d_\e)-\mathfrak J_\e(p_\e, d_0)=\zeta_n(\e)\left(\mathtt G_{p_\e}(d_\e)-\mathtt G_{p_\e}(d_0)+o(1)\right)\leq -\gamma\zeta_n(\e)+o(\zeta_n(\e))<0 \end{equation} for any $\e$ sufficiently small which is again a contradiction.\\ It remains to show that $(p_\e, d_\e)\to (p_0, d_0)$ as $\e\to 0$. Indeed,
by using the fact that $(p_\e, d_\e)$ is a maximum point for $\mathfrak J_\e$ and Proposition \ref{pc} we get
$$\mathfrak J_\e(p_0, d_\e)\leq \mathfrak J_\e(p_\e, d_\e) \quad\iff\quad \mathfrak E(p_0)\leq \mathfrak E(p_\e).$$ Moreover by \eqref{Ep} $$\mathfrak E(p_\e)\leq \mathfrak E(p_0)$$ and hence, passing to the limit it follows $$\lim_{\e\to 0}\mathfrak E(p_\e)=\mathfrak E(p_0).$$ Up to a subsequence, since $p_\e$ is a local maximum for $\mathfrak E$ it follows that $p_\e\to p_0$.\\
In the same way one can show that $d_\e\to d_0$ as $\e\to 0$.
\end{proof}

Finally, we evaluate the energy of the bubble.
\begin{lemma}\label{energiabubble}
The energy of the bubble is:
$$\begin{aligned} \mathfrak E(p)&:=\frac{a_n}{|K(p)|^{\frac{n-2}{2}}}\left[-(n-1) \varphi_{\frac{n+1}{2}}(p)+\frac{\mathfrak D_n(p)}{(\mathfrak D_n^2(p)-1)^{\frac{n-1}{2}}}\right]\end{aligned}$$ where $$a_n:=\alpha_n^\dsh \omega_{n-1}I^n_{n-1}\frac{n-3}{(n-1)\sqrt{n(n-1)}}.$$\end{lemma}\begin{proof}
The energy of the bubble is given by
$$E(U)=\frac{c_n}{2}\int_{\mathbb R^+_n}|\nabla U|^2+\frac{n-2}{2n}|K(p)|\int_{\mathbb R^n_+}U^{2^*}-(n-2)H(p)\int_{\partial\mathbb R^n_+}U^{\dsh}$$
where $c_n:=\frac{4(n-1)}{n-2}$ and $U$ is defined in \eqref{U}.\\ We recall that $U$ satisfies \begin{equation*}
\left\{\begin{aligned}& -c_n \Delta U =-|K(p)|U^{\frac{n+2}{n-2}}\quad &\mbox{in}\,\, \mathbb R^n_+\\ &\frac{2}{n-2}\frac{\partial U}{\partial\nu}=H(p) U^{\frac{n}{n-2}}\quad &\mbox{on}\,\,\partial\mathbb R^n_+.\end{aligned}\right.\end{equation*}
Hence
$$\frac{c_n}{2}\int_{\mathbb R^n_+}|\nabla U|^2=\frac{c_n(n-2)}{4}H(p)\int_{\partial\mathbb R^n_+}U^\dsh-\frac 12 |K(p)|\int_{\mathbb R^n_+}U^{\dst}$$
Then
$$\begin{aligned}E(U)&=-\frac 1 n |K(p)|\int_{\mathbb R^n_+}U^\dst+H(p)\int_{\partial\mathbb R^n_+}U^\dsh\\
&=-\frac 1 n |K(p)| \frac{\alpha_n^\dst}{|K(p)|^{\frac n 2}}\int_{\mathbb R^n_+}\frac{1}{(|\tilde x|^2+(x_n+\mathfrak D_n(p))^2-1)^n}\, d\tilde x\, dx_n\\
&+H(p)\frac{\alpha_n^\dsh}{|K(p)|^{\frac{n-1}{2}}}\int_{\partial\mathbb R^n_+}\frac{1}{(|\tilde x|^2+\mathfrak D_n^2(p)-1)^{n-1}}\, d\tilde x\end{aligned}$$
Now, by making some computations we get that for $n\geq 4$
$$\begin{aligned} \int_{\partial\mathbb R^n_+}\frac{1}{(|\tilde x|^2+\mathfrak D_n^2(p)-1)^{n-1}}\, d\tilde x
=\omega_{n-1}\frac{n-3}{n-1}\frac{I^{n}_{n-1}}{(\mathfrak D_n(p)^2-1)^{\frac{n-1}{2}}}.\end{aligned}$$


and  
$$\begin{aligned}
\int_{\mathbb R^n_+}\frac{1}{(|\tilde x|^2+(x_n+\mathfrak D_n(p))^2-1)^n}\, d\tilde x\, dx_n
=\omega_{n-1}\frac{n-3}{2(n-1)}I_{n-1}^{n}\varphi_{\frac{n+1}{2}}(p).
\end{aligned}$$
Putting together the terms we get
$$\begin{aligned}\mathfrak E(p)=E(U)&=-\frac{\alpha_n^\dst(n-3)}{2 n(n-1)}\omega_{n-1}I^n_{n-1}\frac{\varphi_{\frac{n+1}{2}}(p)}{|K(p)|^{\frac{n-2}{2}}}+\frac{\alpha_n^{\dsh}(n-3)}{n-1}\omega_{n-1}I^{n}_{n-1}\frac{H(p)}{|K(p)|^{\frac{n-1}{2}}(\mathfrak D_n^2(p)-1)^{\frac{n-1}{2}}}\\
&=\alpha_n^\dsh \omega_{n-1}I^n_{n-1}\frac{n-3}{n-1}\frac{1}{|K(p)|^{\frac{n-2}{2}}}\left[-\frac{ \alpha_n^{\dst-\dsh}}{2 n}\varphi_{\frac{n+1}{2}}(p)+\frac{H(p)}{|K(p)|^{\frac{1}{2}}(\mathfrak D_n^2(p)-1)^{\frac{n-1}{2}}}\right]\\
&=\underbrace{\alpha_n^\dsh \omega_{n-1}I^n_{n-1}\frac{n-3}{(n-1)\sqrt{n(n-1)}}}_{:=a_n}\frac{1}{|K(p)|^{\frac{n-2}{2}}}\left[-(n-1)\varphi_{\frac{n+1}{2}}(p)+\frac{\mathfrak D_n(p)}{(\mathfrak D_n^2(p)-1)^{\frac{n-1}{2}}}\right].\\
\end{aligned} $$\end{proof}

\end{document}